\documentclass[reqno]{amsart}
\usepackage{amssymb}
\usepackage{hyperref}
\usepackage{xcolor}
\usepackage{longtable}

\numberwithin{equation}{section}
\newtheorem{theorem}{Theorem}[section]
\newtheorem{lemma}[theorem]{Lemma}

\newtheorem{corollary}[theorem]{Corollary}
\newtheorem*{remark}{Remark}

\newtheorem{definition}[theorem]{Definition}
\newtheorem{example}[theorem]{Example}
\begin{document}

\title[Sharp inequalities and solitons for statistical submersions]{Sharp inequalities and solitons for statistical submersions}
\author{Mohd. Danish Siddiqi}
\address{ Department of Mathematics\\College of Science\\ Jazan University\\
  Jazan\\ Kingdom of Saudi Arabia}
\email{msiddiqi@jazanu.edu.sa}

\author{Aliya Naaz Siddiqui}
\address{Division of Mathematics\\ School of Basic $\&$ Applied Science\\ Galgotias University\\ Greater Noida 203201\\ Uttar Pradesh}
\email{aliyanaazsiddiqui9@gmail.com}

\author{Bang-Yen Chen}
\address{Department of Mathematics\\ Michigan State University\\ 619 Red Cedar Road \\ East Lansing, MI 48824-1027, USA}
\email{chenb@msu.edu}

\begin{abstract}
In this research article, initially, we  prove some sharp inequalities on statistical submersions involving Ricci and scalar curvatures of the statistical manifolds. In addition, we establish the geometrical bearing on statistical submersions in terms of Ricci-Bourguignon soliton. Moreover, we characterize the fibers of a statistical submersion as Ricci-Bourguignon solitons with conformal vector field. Finally, in the particular case when the vertical potential vector field of the Ricci-Bourguignon soliton is of gradient type, we derive a Poisson equation for a statistical submersion.
\end{abstract}

\subjclass{53C25, 53C43}

\keywords{Ricci-Bourguignon soliton, Statistical submersion, Statistical manifold, Einstein manifold}

\maketitle

\section{Introduction}\label{intro}

The concept of submersions between Riemannian manifolds was first described by O'Neill \cite{O} and  Gray \cite{Gr}. In 1983, Watson \cite{W1} popularized the knowledge of Riemannian submersions between almost Hermitian manifolds under the name of almost Hermitian submersions. Afterwards, almost Hermitian submersions have
been investigated between various subclasses of almost Hermitian manifolds.
Also, Riemannian submersions were defined to many subclasses of almost
contact metric manifolds under the name of contact Riemannian submersions \cite{Sa2}.
Most of the studies related to Riemannian, almost Hermitian or contact
Riemannian submersions are available in the books \cite{Fa, Sa1}.

In 1981, the concept of Ricci-Bourguignon flow as an extension of Ricci flow \cite{Hami} has been initiated by J. P. Bourguignon \cite{ref7} based on some unprinted work of Lichnerowicz and a paper of Aubin \cite{ref1}. Ricci-Bourguignon flow  is an intrinsic geometric flow on Riemannian manifolds, whose fixed points are solitons.
Therefore, Ricci-Bourguignon solitons generate self-similar solution to the Ricci-Bourguignon flow \cite{ref7}
\begin{equation}\label{rb1}
\frac{\partial g}{\partial t}=-2(Ric-\rho Rg), \quad g(0)=g_{0},
\end{equation}
where $Ric$ is the Ricci curvature tensor, $R$ is the scalar curvature with respect to the $g$ and $\rho$ is a real non-zero constant.  It should be noticed that for special values of the constant $\rho$ in equation (\ref{rb1}), we obtain the following situations for the tensor $Ric-\rho Rg$ appearing in equation (\ref{rb1}). This PDE system (\ref{rb1}) defined by the evolution equation is of special interest, in particular \cite{ref7}, we have
\begin{enumerate}
\item[(i)] for $\rho=\frac{1}{2}$, the Einstein tensor $Ric-\frac{R}{2}g$ (and Einstein soliton \cite{Bla2}),
\item[(ii)] for $\rho=\frac{1}{n}$, the traceless Ricci tensor $Ric-\frac{R}{n}g$,
\item[(iii)] for $\rho=\frac{1}{2(n-1)}$, the Schouten tensor $Ric-\frac{R}{2(n-1)}g$ (and Schouten soliton \cite{ref7}),
\item[(iv)] for $\rho=0$, the Ricci tensor $Ric$ (and Ricci soliton \cite{Hami}).
\end{enumerate}
In dimension two, the first three tensors are zero, hence the flow is static, and in higher dimensions, the value of $\rho$ is strictly ordered as above, in descending order.

Short time existence and uniqueness for the
solution of this geometric flow have been proved in \cite{ref7}. In fact, for sufficiently small
$t$, the equation (\ref{rb1}) has a unique solution for $\rho < \frac{1}{2(n-1)}$.

On the other hand, quasi-Einstein metrics or Ricci solitons serve as  solutions to Ricci
flow equation \cite{ref8}. This motivates the study of a more general type of Ricci soliton by considering
the Ricci-Bourguignon flow \cite{ref9, Dan3, Dan1}. In fact, a Riemannian manifold of dimension
$n\geq 3$ is said to be a \textit{Ricci-Bourguignon soliton} \cite{ref1} if
\begin{eqnarray}\label{rb2}
\frac{1}{2}\mathcal{L}_{V}g+Ric+(\lambda-\rho R)g=0,
\end{eqnarray}
where $\mathcal{L}_{V}$ denotes the Lie derivative operator along the vector field $V$, $\rho$ is a non-zero constant and $\lambda$ is an arbitrary real constant. Similar to Ricci soliton, a Ricci-Bourguignon soliton $(\mathcal{M},g,V,\lambda)$ is called expanding if
$\lambda > 0$, steady if $\lambda= 0$, and shrinking if $\lambda< 0$.

We say that $(M,g, f,\lambda)$ is a gradient Ricci-Bourguignon soliton if the potential vector field $V$ is the gradient of some smooth function $f$ on $M$. In this case, the soliton equation is
\begin{equation}\label{rb3}
Hess(f) +Ric+(\lambda-{\rho R})g=0,
\end{equation}
where {\em Hess(f)} is the Hessian of $f$.

In 2019, Meri\c{c} et al. introduced and studied Riemannian submersions whose total manifolds admit a Ricci soliton \cite{semsi}. In \cite{aliya2}, M. D. Siddiqi, et al. studied Conformal $\eta$-Ricci solitons on Riemannian submersions under canonical variation. Barndorff-Nielsen and Jupp \cite{Bar} discussed Riemannian submersions from the viewpoint of statistics. Takano \cite{Taka1} put forward the discourse of the statistical submersions of various spaces. In \cite{aliya1}, statistical solitons have been studied by Siddiqui et al. for warped product statistical submanifolds. Recently, Blaga and Chen also studied some properties of gradient solitons in statistical manifolds which is closely related to this topic \cite{Bla3}. Therefore, in this paper we will study statistical submersions whose total space admits  Ricci-Bourguignon solitons and gradient Ricci-Bourguignon solitons.

Furthermore, finding the relationship between  extrinsic and intrinsic invariants of a submanifold has been a very popular problem in the past twenty five years. The study in this direction  was initiated in 1993 by Chen in  \cite{Chen} via his $\delta$-invariants (see also \cite{Book11}). In \cite{mgk}, Meri\c{c} et al. established a series of inequalities for Riemannian submersions. Pursuing the development of this research, we establish in this paper some additional relations between the main intrinsic invariants and the main extrinsic invariants of the vertical and the horizontal spaces of a statistical manifold which admits a statistical submersion.

\section{Background of Statistical Submersions}\label{sect_2}

In this section, we provide the necessary background for understanding statistical submersions.

Let $\mathcal{M}$ and $\mathcal{N}$ be two Riemannian manifolds and  $\psi:\mathcal{M}\longrightarrow\mathcal{N}$ be a Riemannian submersion such that all the fibers are Riemannian submanifolds of $\mathcal{M}$ and $\psi_{\ast}$ preserves the lengths of horizontal vectors (for more details, see \cite{Sa1, Fa}). In 2001, Abe and Hasegawa \cite{Abe} investigated affine submersions with horizontal distribution from statistical manifolds. Furthermore, statistical submersions were also discussed by Takano in \cite{Taka1, Taka2}.

Let $(\mathcal{M},g)$ be a Riemannian manifold with a non-degenerate metric $g$ and let $\nabla$ be a torsion-free affine connection on $\mathcal{M}$. The triplet $(\mathcal{M},\nabla,g)$ is called a {\em statistical manifold} if it satisfies ($\nabla_{U}g)(V,W)=(\nabla_{V}g)(U,W)$ \cite{Taka1}.
For a statistical manifold $(\mathcal{M},\nabla,g)$, we consider a second connection ${\nabla}^{\ast}$ satisfying
\begin{equation}\label{ss1}
Wg(U,V)=g(\nabla_{W}U,V)+g(U,{\nabla}^{\ast}_{W} V),
\end{equation}
for any vector fields $U,V,W$ on $\mathcal{M}$ and call ${\nabla}^{\ast}$ the {\em conjugate} (or {\em dual}) connection of ${\nabla}$ with respect to $g$. The affine connection ${\nabla}^{\ast}$ is torsion-free, $({\nabla}^{\ast}_{U}g)(V,W)=({\nabla}^{\ast}_{V}g)(U,W)$ and obeys ${({\nabla}^{\ast})}^{\ast}=\nabla$. Certainly, $(\mathcal{M},{\nabla}^{\ast},g)$ is also a statistical manifold. For example, every Riemannian manifold $(\mathcal{M},\nabla,g)$ with its Riemannian connection $\nabla$ is a {\em trivial statistical manifold}. In this case, if $Rie$ and $Rie^{\ast}$ stand for the curvature tensors on $\mathcal{M}$ with respect to the affine connection $\nabla$ and its conjugate ${\nabla}^{\ast}$, respectively, then we turn up
\begin{equation}\label{ss2}
g(Rie(U,V)W,X)=-g(W,{Rie}^{\ast}(U,V)X),
\end{equation}
for any vector fields $U,V,W,X$ on $\mathcal{M}$ \cite{Taka1}.

Let $\psi: ({\mathcal{M}}^{p},g)\longrightarrow (\mathcal{N}^{q},\hat{g})$ be a Riemannian submersion between two Riemannian manifolds $({\mathcal{M}}^{p},g)$ and  $(\mathcal{N}^{q},\hat{g})$. For any point $ x\in\mathcal{N}$, the Riemannian submanifold $\psi^{-1}(x)$ is $(p-q)$-dimensional with the induced metric $\bar{g}$ called a {\em fiber} and denoted by $\bar{\mathcal{M}}$. In the tangent bundle $T\mathcal{M}$ of $\mathcal{M}$, the vertical and horizontal distributions are denoted by $\mathcal{V}(\mathcal{M})$ and $\mathcal{H}(\mathcal{M})$, respectively. So, we have
\begin{equation}\nonumber
T_{r}(\mathcal{M})=\mathcal{V}_{r}(\mathcal{M})\oplus\mathcal{H}_{r}(\mathcal{M}),\;\; r\in \mathcal{M}.
\end{equation}

We call a vector field $U$ on $\mathcal{M}$ \textit{projectable} if there exists a vector field $U_{\ast}$ on $\mathcal{N}$ such that $\psi_{\ast}(U_{r})=U_{{\ast}\psi({r})}$, for each $r\in \mathcal{M}$. In this case, $U$ and $U_{\ast}$ are called \textit{$\psi$-related}. Also, a vector field $U$ on $\mathcal{H}(\mathcal{M})$ is called {\em basic} if it is projectable \cite{O}.

If $U$ and $V$ are basic vector fields, $\psi$-related to $\hat{U}$ and $\hat{V}$, we have the following facts
\begin{enumerate}
\item{} $g(U,V)=\hat{g}(\hat{U},\hat{V})\circ\psi$,
\item{} $\mathcal{H}[U,V]$ is a basic vector field $\psi$-related to $[X_{\ast},Y_{\ast}]$,
\item{} $\mathcal{H}(\nabla_{U}V)$ is a basic vector field $\psi$-related to $\hat{\nabla}_{\hat{U}}\hat{V}$.
\end{enumerate}

The geometry of Riemannian
submersions is characterized by O'Neill's tensors $\mathcal{T}$ and
$\mathcal{A}$, defined as follows \cite{O}
\begin{equation}\label{e1}
\mathcal{T}_{E}F=\mathcal{V}\nabla_{\mathcal{V}E}\mathcal{H}F+\mathcal{H}\nabla_{\mathcal{V}E}\mathcal{V}F,
\end{equation}
\begin{equation}\label{e2}
\mathcal{A}_{E}F=\mathcal{V}\nabla_{\mathcal{H}E}\mathcal{H}F+\mathcal{H}\nabla_{\mathcal{H}E}\mathcal{V}F,
\end{equation}
for any vector fields $E$, $F$ on $\mathcal{M}$. It is easy to see
that $\mathcal{T}_{E}$ and $\mathcal{A}_{E}$ are skew-symmetric
operators with respect to $g$ on the tangent bundle of $\mathcal{M}$ reversing the vertical and
the horizontal distributions. We summarize the properties of the
tensor fields $\mathcal{T}$ and $\mathcal{A}$. Let $V,W$ be
vertical and $X,Y$ be horizontal vector fields on $\mathcal{M}$. Then, we
have
\begin{equation}\label{e3}
\mathcal{T}_{V}W=\mathcal{T}_{W}V,
\end{equation}
\begin{equation}\label{e4}
\mathcal{A}_{X}Y=-\mathcal{A}_{Y}X=\frac{1}{2}\mathcal{V}[X,Y].
\end{equation}

Let $(\mathcal{M},\nabla,g)$ be a statistical manifold and $\psi:\mathcal{M}\longrightarrow\mathcal{N}$ be a Riemannian submersion. We denote the dual affine connections on any fiber $\bar{\mathcal{M}}$ of $\psi$ by ${\bar{\nabla}}$ and ${\bar{\nabla}}^{\ast}$. It can be simply observed that ${\bar{\nabla}}$ and ${\bar{\nabla}}^{\ast}$ are torsion-free and conjugate to each other with respect to $g$.

%Notice that ${\bar{\nabla}}_{E}F$ and ${\bar{\nabla}}^{\ast}_{E}{F}$ are well-defined vertical vector fields on $\mathcal{M}$ for vertical vector fields $E$ and $F$ on $\mathcal{M}$, more precisely ${\bar{\nabla}}_{E}F=\mathcal{V}{{\nabla}}_{E}{F}$ and ${\bar{\nabla}}^{\ast}_{E}{F}=\mathcal{V}{{\nabla}}^{\ast}_{E}{F}$.

%We put $S=\nabla-{{\nabla}}^{\ast}$, which is a symmetric tensor field.

A submersion $\psi:(\mathcal{M},\nabla, g)\longrightarrow(\mathcal{N},\hat{\nabla},\hat{g})$ between two statistical manifolds is called a {\em statistical submersion} if $\psi$ satisfies $\psi_{\ast}(\nabla_{X}Y)_{r}=(\hat{{\nabla}}_{\psi_{\ast} (X)} {\psi_{\ast}(Y)})_{\psi(r)}$ for basic vector field $X,Y$ on $\mathcal{M}$ and $r\in{\mathcal{M}}$. Changing $\nabla$ for $\nabla^{\ast}$ in the above expressions, we define $\mathcal{T}^{\ast}$ and $\mathcal{A}^{\ast}$ \cite{Taka1}. $\mathcal{A}$ and $\mathcal{A}^{\ast}$ are equal to zero if and only if $\mathcal{H}(\mathcal{M})$ is integrable with respect to $\nabla$ and $\nabla^{\ast}$, respectively. For $X,Y\in \Gamma\mathcal{H}(\mathcal{M})$ and $V,W\in \Gamma\mathcal{V}(\mathcal{M})$, we turn up
\begin{equation}\label{ss3}
g(\mathcal{T}_{V}W,X)=-g(W,\mathcal{T}^{\ast}_{V}X),\quad \quad g(\mathcal{A}_{X}Y, V)=-g(Y,\mathcal{A}^{\ast}_{X}V).
\end{equation}

\section{Properties of Statistical Submersions}\label{sect_3}

In this section, we will discuss some useful  properties of statistical submersions proposed by Takano \cite{Taka1}. In the same paper, Takano provided the following lemmas which are useful for this study. Therefore, for a statistical submersion $\psi:(\mathcal{M},\nabla, g)\longrightarrow(\mathcal{N},\hat{\nabla},\hat{g})$, we have \cite{O, Taka1}
\begin{lemma}\label{L1}
  \cite{Taka1} If $X$ and $Y$ are horizontal vector fields, then $\mathcal{A}_{X}Y=-\mathcal{A}^{\ast}_{Y}X$.
\end{lemma}

\begin{lemma}
\cite{Taka1} For $X,Y\in \Gamma\mathcal{H}(\mathcal{M})$ and $V,W\in \Gamma\mathcal{V}(\mathcal{M})$, we have
\begin{equation}\label{e5}
\nabla_{V}W=\mathcal{T}_{V}W+\bar{\nabla}_{V}W,\quad\quad {\nabla}^{\ast}_{V}W=\mathcal{T}^{\ast}_{V}W+{\bar{\nabla}}^{\ast}_{V}W,
\end{equation}
\begin{equation}\label{e6}
\nabla_{V}X=\mathcal{T}_{V}X+\mathcal{H}\nabla_{V}X,\quad\quad {\nabla}^{\ast}_{V}X=\mathcal{T}^{\ast}_{V}X+\mathcal{H}{\nabla^{\ast}}_{V}X,
\end{equation}
\begin{equation}\label{e7}
\nabla_{X}V=\mathcal{A}_{X}V+\mathcal{V}\nabla_{X}V,\quad\quad {\nabla}^{\ast}_{X}V=\mathcal{A}^{\ast}_{X}V+\mathcal{V}{\nabla}^{\ast}_{X}V,
\end{equation}
\begin{equation}\label{e8}
\nabla_{X}Y=\mathcal{H}\nabla_{X}Y+\mathcal{A}_{X}Y,\quad\quad {\nabla}^{\ast}_{X}Y=\mathcal{H}{\nabla}^{\ast}_{X}Y+\mathcal{A}^{\ast}_{X}Y.
\end{equation}
Furthermore, if $X$ is basic, then $\mathcal{H}\nabla_{V}X=\mathcal{A}_{X}V$ and $\mathcal{H}{\nabla}^{\ast}_{V}X=\mathcal{A}^{\ast}_{X}V$.
\end{lemma}

Let $Rie$ (resp. $Rie^{*}$) be the curvature tensor with respect to $\nabla$ (resp. $\nabla^{*}$) of $\mathcal{M}$. The curvature tensor with respect to the induced affine connection $\bar{\nabla}$ (resp. ${\bar{\nabla}^{*}}$) of each fiber are given by $\bar{Rie}$ (resp. $\bar{Rie}^{\ast}$). Moreover,  $\hat{Rie}(X, Y)Z$ (resp. $\hat{Rie}^{*}(X, Y)Z$ ) are horizontal vector fields such that
$$\psi_{\ast}(\hat{Rie}(X, Y)Z) = \hat{Rie}(\psi_{\ast}X, \psi_{\ast}Y)\psi_{\ast}Z$$$$\text{(resp. } \; \psi_{\ast}(\hat{Rie}^{\ast}(X, Y)Z) = \hat{Rie}^{\ast}(\psi_{\ast}X, \psi_{\ast}Y)\psi_{\ast}Z\; \text{)}$$
at each $p \in M$, where $\hat{Rie}$ (resp. $\hat{Rie}^{\ast}$) denotes the curvature tensor with respect to the affine connection $\hat{\nabla}$ (resp. $\hat{\nabla}^{\ast}$). Then we have the following theorem \cite{Taka1}

\begin{theorem}\label{TS1}
If $\psi:(\mathcal{M},\nabla, g)\longrightarrow(\mathcal{N},\hat{\nabla},\hat{g})$ is a statistical submersion, then for $E,F,G,K\in \Gamma\mathcal{V}(\mathcal{M})$ and $X,Y,Z,W\in \Gamma\mathcal{H}(\mathcal{M})$, we have
\begin{align}\label{c1}
Rie(E,F,G,K)= & \bar{Rie}(E,F,G,K)+g(\mathcal{T}_{E}G,\mathcal{T}^{\ast}_{F}K)-g(\mathcal{T}_{F}G,\mathcal{T}^{\ast}_{E}K),\\
\label{c2}
Rie(X,Y,Z,W)= & \hat{Rie}(X,Y,Z,W)+g((\mathcal{A}_{X}+\mathcal{A}^{\ast}_{X})Y,\mathcal{A}^{\ast}_{Z}W)\\
\nonumber
& -g(\mathcal{A}_{Y}Z,\mathcal{A}^{\ast}_{X}W)+g(\mathcal{A}_{X}Z,\mathcal{A}^{\ast}_{Y}W).
\end{align}
\end{theorem}

Now, we get
\begin{theorem}\label{p3}
The Ricci curvature tensors of $(\mathcal{M},\nabla, g)$, $(\mathcal{N}, \hat{\nabla},\hat{g})$ and of any fiber of $\psi$ denoted by $Ric$, $\hat{Ric}$ and $\bar{Ric}$, respectively, satisfy
\begin{align}\label{c3}
Ric(E,F)=&\bar{Ric}(E,F)-g(\mathcal{T}_{E}F,N^{\ast})+\sum^{n}_{i=1}g((\nabla_{X_{i}}\mathcal{T})({E},F),X_{i}) \\
\nonumber
&+g(\mathcal{A}_{X_{i}}E,\mathcal{A}^{\ast}_{X_{i}}F)-\sum^{m}_{j=1}g((\nabla^{\ast}_{E_{j}}\mathcal{A})({X_{i}},X),V),\\
\label{c4}
Ric(X,Y)=&\hat{Ric}(X,Y)+g(\nabla^{\ast}_{X}N^{\ast},Y)-\sum^{m}_{j=1}g(\mathcal{T}_{E_{j}}{X},\mathcal{T}_{E_{j}}Y)\\
\nonumber
&+\sum^{n}_{i=1}g(\mathcal{A}_{X}{X_{i}},\mathcal{A}^{\ast}_{Y}{X_{i}})+\sum^{n}_{i=1}g(\nabla_{X_{i}}\mathcal{A})({X_{i}}, X),Y)\\
\nonumber
&+\sum^{n}_{i=1}g(\mathcal{A}_{X_i}X_{i}, \mathcal A_{X}Y)-g(\mathcal{A}^{\ast}_{X}X_{i},\mathcal{A}^{\ast}_{X}X_{i}),
\end{align}
for any $E,F\in\Gamma \mathcal{V}(\mathcal{M})$ and $X,Y\in\Gamma \mathcal{H}(\mathcal{M})$, where $\left\{X_{i}\right\}_{1\leq i\leq n}$ and $\left\{E_{j}\right\}_{1\leq i\leq m}$ are orthonormal basis of $\mathcal{H} (horizontal)$ and $\mathcal{V} (vertical)$ distributions, respectively.
\end{theorem}
\indent
For any fiber of the statistical submersion $\psi$, the mean curvature vector field $H$ is given by $mH=N$, where
\begin{equation}\label{c7}
N=\sum^{m}_{j=1}\mathcal{T}_{E_{j}}E_{j},
\end{equation}
$m$ is the dimension of any fiber of $\psi$ and $\left\{E_{1},E_{2},\dots,E_{m}\right\}$ is an orthonormal basis in the vertical distribution. We point out  the horizontal vector field $N$ vanishes if and only if any fiber of the statistical submersion $\psi$ is minimal.

Now, from (\ref{c7}), we find
\begin{equation}\label{c8}
g(\nabla_{U}N,X)=\sum^{m}_{j=1}g((\nabla_{U}\mathcal{T})(E_{j},E_{j}),X),
\end{equation}
for any $U\in\Gamma(TM)$ and $X\in\Gamma \mathcal{H}(\mathcal{M})$.

Also, for any tensor field $\mathcal{P}$, we put
$$\hat{\delta}\mathcal{P} = - \sum^{n}_{i=1} (\nabla_{X_{i}}\mathcal{P})_{X_{i}}\ \textrm{ and }\ \bar{\delta}\mathcal{P} = - \sum^{m}_{j=1} (\nabla_{E_{j}}\mathcal{P})_{E_{j}}.$$

The horizontal divergence of any vector field $X$ in $\Gamma \mathcal{H}(\mathcal{M})$ is denoted by $\delta(X)$ and is given by
\begin{equation}\label{cc}
\delta(X)=\sum^{n}_{i=1}g(\nabla_{X_{i}}X,X_{{i}}),
\end{equation}
where $\left\{X_{1},X_{2},\dots,X_{n}\right\}$ is an orthonormal basis of the horizontal space $\Gamma \mathcal{H}(\mathcal{M})$.

Hence, considering (\ref{cc}), we have
\begin{equation}\label{cc1}
\delta(N)=\sum^{n}_{i=1}\sum^{m}_{j=1}g((\nabla_{X_{i}}\mathcal{T})(E_{j},E_{j}),X_{i}).
\end{equation}

\section{Some Sharp Inequalities Involving Ricci and Scalar Curvatures}\label{sect_4}

This section is mainly devoted to the study of  Ricci and scalar curvatures for a given  statistical submersion and  we construct some inequalities involving the Ricci and scalar curvatures. From Theorem \ref{p3} and \cite{Taka1}, we have
\begin{align}\label{1c3}
Ric(E,F)=&\bar{Ric}(E,F)-g(\mathcal{T}_{E}F,N^{\ast})+ (\hat{\delta}\mathcal{T})(E, F)\\
\nonumber
&+g(\mathcal{A}E,\mathcal{A}^{\ast}F)-g(\nabla^{\ast}_{E}\sigma, F)
\end{align}
and
\begin{align}\label{1c4}
Ric(X,Y)=&\hat{Ric}(X,Y)+g(\nabla^{\ast}_{X}N^{\ast},Y)-g(\mathcal{T}X, \mathcal{T}Y) + (\hat{\delta}\mathcal{A})(X, Y)\\
\nonumber
& + g(\sigma, \mathcal{A}_{X} Y) - g(\mathcal{A}_{X}, \mathcal{A}^{\ast}_{Y}) - g(\mathcal{A}^{\ast}_{X}, \mathcal{A}^{\ast}_{Y}),
\end{align}
where
\begin{equation} \nonumber
(\hat{\delta}\mathcal{T})(E, F) = \sum^{n}_{i=1}g((\nabla_{X_{i}}\mathcal{T})({E},F),X_{i}),
\end{equation}
\begin{equation}\nonumber
(\hat{\delta}\mathcal{A})(X, Y) = \sum^{m}_{j=1}g((\nabla_{E_{j}}\mathcal{A})({X},Y),E_{j}),
\end{equation}
\begin{equation}\nonumber
g(\mathcal{A}_{X}, \mathcal{A}_{Y}) = \sum^{n}_{i = 1}g(\mathcal{A}_{X}X_{i}, \mathcal{A}_{Y}X_{i}) = \sum^{m}_{j = 1}g(\mathcal{A}^{\ast}_{X}E_{j}, \mathcal{A}^{\ast}_{Y}E_{j}), \ \sigma = \sum^{n}_{i =1} \mathcal{A}_{X_{i}}X_{i},
\end{equation}
\begin{equation}\nonumber
g(\mathcal{A}E, \mathcal{A}F) = \sum^{n}_{i = 1}g(\mathcal{A}_{X_{i}}E, \mathcal{A}_{X_{i}}F), \quad g(\mathcal{T}X, \mathcal{T} Y) = \sum^{m}_{j = 1}g(\mathcal{T}_{E_{j}}X, \mathcal{T}_{E_{j}}Y).
\end{equation}

Thus, we derive the following inequality:

\begin{theorem}
 Let $\psi:(\mathcal{M},\nabla, g)\longrightarrow(\mathcal{N},\hat{\nabla},\hat{g})$ be a statistical submersion. Then, we have
\begin{equation}\label{f}
Ric(E,E) \geq \bar{Ric}(E,E)- m^{2}g(\mathcal{T}_{E}E,H^{\ast})+ (\hat{\delta}\mathcal{T})(E, E) - (\bar{\delta}\sigma)(E, E).
\end{equation}
The equality case holds in (\ref{f}) if and only if $\mathcal{H}(\mathcal{M})$ is integrable.
\end{theorem}

Since $2\mathcal{A}^{0} = \mathcal{A} + \mathcal{A}^{\ast}$, we have

\begin{theorem}
 Let $\psi:(\mathcal{M},\nabla, g)\longrightarrow(\mathcal{N},\hat{\nabla},\hat{g})$  be a statistical submersion. Then, we have
\begin{align}\label{1f}
Ric(X,X) \leq &\hat{Ric}(X,X)+g(\nabla^{\ast}_{X}N^{\ast},X) + (\hat{\delta}\mathcal{A})(X, X)\\
\nonumber
 &+ g(\sigma, \mathcal{A}_{X} X) - 2g(\mathcal{A}^{0}_{X}, \mathcal{A}^{\ast}_{X}).
\end{align}
The equality case holds in (\ref{1f}) if and only if each fiber is totally geodesic with respect to $\nabla$ ($\mathcal{T} = 0$).
\end{theorem}

In view of (\ref{c1}), using the symmetry of $\mathcal{T}$ and $\mathcal{T}^{\ast}$, we derive
\begin{equation}\nonumber
 2R = 2\bar{R} - m^{2}g(H, H^{\ast}) + \sum^{m}_{i,j=1}g(\mathcal{T}_{E_{i}}E_{j}, \mathcal{T}^{\ast}_{E_{i}}E_{j}).
 \end{equation}

\begin{theorem}
Let $\psi:(\mathcal{M},\nabla, g)\longrightarrow(\mathcal{N},\hat{\nabla},\hat{g})$  be a statistical submersion. Then, we have
 \begin{equation}\label{3}
 2R \geq 2\bar{R} - m^{2}g(H, H^{\ast}).
 \end{equation}
The equality case holds in (\ref{3}) if and only if either $\mathcal{T}$ or $\mathcal{T}^{\ast}$ is a multiple of the other. In particular, either each fiber is totally geodesic with respect to $\nabla$ ($\mathcal{T} = 0$)  or each fiber is totally geodesic with respect to $\nabla^{\ast}$ ($\mathcal{T}^{\ast} = 0$).
\end{theorem}

In view of (\ref{c2}), using $\mathcal{A}_{X}Y = - \mathcal{A}_{Y}^{\ast}X$ for horizontal vector fields $X$ and $Y$, we arrive at
\begin{align}\nonumber
 2R =& 2\hat{R} + \sum_{i,j=1}^{n} \bigg[g(\mathcal{A}_{X_{i}}X_{j},\mathcal{A}^{\ast}_{X_{j}}X_{i}) - g(\mathcal{A}_{X_{j}}X_{j},\mathcal{A}^{\ast}_{X_{i}}X_{i})\\
 \nonumber
&+ g((\mathcal{A}_{X_{i}} +\mathcal{A}^{\ast}_{X_{i}})X_{j}, \mathcal{A}_{X_{j}}^{\ast}X_{i})\bigg]\\
\nonumber
=& 2\hat{R} + g(\sigma, \sigma) - \sum_{i,j=1}^{n} \bigg[2g(\mathcal{A}_{X_{i}}X_{j},\mathcal{A}_{X_{i}}X_{j}) + g(\mathcal{A}^{\ast}_{X_{i}} X_{j}, \mathcal{A}_{X_{i}}X_{j})\bigg] .
\end{align}

\begin{theorem}
Let $\psi:(\mathcal{M},\nabla, g)\longrightarrow(\mathcal{N},\hat{\nabla},\hat{g})$  be a statistical submersion. Then, we have
 \begin{equation}\label{4}
 2R \leq 2\hat{R} + g(\sigma, \sigma).
 \end{equation}
The equality case holds in (\ref{4}) if and only if $\mathcal{A}_{\mathcal{H}}\mathcal{H} = 0$.
\end{theorem}

Taking into account relations (\ref{1c3}) and (\ref{1c4}), we have the following equation \cite{Taka1}
\begin{align}\label{1}
R - \bar{R}-\hat{R} = & - 2g(\mathcal{A}, \mathcal{A}) + g(\mathcal{A}, \mathcal{A}^{\ast}) - g(\mathcal{T}, \mathcal{T}^{\ast}) - g(N, N^{\ast})\\
\nonumber
& - \hat{\delta}N -  \hat{\delta}^{\ast}N^{\ast} - \bar{\delta}\sigma + \bar{\delta}^{\ast}\sigma + g(\sigma, \sigma),
\end{align}
where $\bar{R}$ and $\hat{R}$ are the scalar curvatures of the vertical and horizontal spaces of $\mathcal{M}$. Here
\begin{equation}\nonumber
g(\mathcal{T}, \mathcal{T}^{\ast}) = \sum^{n}_{i = 1} g(\mathcal{T}X_{i}, \mathcal{T}^{\ast}X_{i}), \quad g(\mathcal{A}, \mathcal{A}) = \sum^{n}_{i=1} g(\mathcal{A}_{X_{i}}, \mathcal{A}_{X_{i}}),
\end{equation}
\begin{equation}\nonumber
g(\mathcal{A}, \mathcal{A}^{\ast}) = \sum^{n}_{i=1} g(\mathcal{A}_{X_{i}}, \mathcal{A}^{\ast}_{X_{i}}).
\end{equation}

By using {\em Cauchy-Buniakowski-Schwarz inequality} and equation (\ref{1}), we have the following theorem:

\begin{theorem}\label{th1}
 Let $\psi:(\mathcal{M},\nabla, g)\longrightarrow(\mathcal{N},\hat{\nabla},\hat{g})$  be a statistical submersion. Then, we have
 \begin{align}\label{2}
 R \geq &\bar{R}+ \hat{R} - 2 ||\mathcal{A}||^{2} + g(\mathcal{A}, \mathcal{A}^{\ast}) - ||\mathcal{T}|| ||\mathcal{T}^{\ast}|| - g(N, N^{\ast})\\
\nonumber
 &- \hat{\delta}N -  \hat{\delta}^{\ast}N^{\ast} - \bar{\delta}\sigma + \bar{\delta}^{\ast}\sigma + g(\sigma, \sigma).
\end{align}
The equality case holds in (\ref{2}) if and only if either $\mathcal{T}$ or $\mathcal{T}^{\ast}$ is a multiple of the other. In particular, either each fiber is totally geodesic with respect to $\nabla$ ($\mathcal{T} = 0$) or each fiber is totally geodesic with respect to $\nabla^{\ast}$ ($\mathcal{T}^{\ast} = 0$).
\end{theorem}

From Theorem \ref{th1}, we have the following corollary:

\begin{corollary}
Let $\psi:(\mathcal{M},\nabla, g)\longrightarrow(\mathcal{N},\hat{\nabla},\hat{g})$ be a statistical submersion. If $\mathcal{H}(\mathcal{M})$ is integrable, then we have
 \begin{equation}\label{2a}
 R \geq \bar{R}+ \hat{R} - ||\mathcal{T}|| ||\mathcal{T}^{\ast}|| - m^{2}g(H, H^{\ast}) - \hat{\delta}N -  \hat{\delta}^{\ast}N^{\ast}.
\end{equation}
The equality case holds in (\ref{2a}) if and only if either $\mathcal{T}$ or $\mathcal{T}^{\ast}$ is a multiple of the other. In particular, either each fiber is totally geodesic with respect to $\nabla$ ($\mathcal{T} = 0$)  or each fiber is totally geodesic with respect to $\nabla^{\ast}$ ($\mathcal{T}^{\ast} = 0$).
\end{corollary}

\begin{example}\label{ex1}
Let ($\mathcal{M}= \left\{(x_{1},...,x_{6})\in\mathbb{R}^{6}\right\},\nabla, g=\sum^{6}_{i,j=1}dx_{i}\otimes dx_{j})$ be a statistical manifold with $\nabla$ given by
\begin{equation}\nonumber
{\nabla}_{e_{1}}e_{1}=e_{6},\quad {\nabla}_{e_{2}}e_{2}=e_{6}, \quad {\nabla}_{e_{3}}e_{3}=e_{6}, \quad{\nabla}_{e_{4}}e_{4}=e_{6}, \quad {\nabla}_{e_{5}}e_{5}=e_{6}
\end{equation}
\begin{equation}\nonumber
{\nabla}_{e_{6}}e_{6}=0, \quad {\nabla}_{e_{6}}e_{i}=0,\quad {\nabla}_{e_{i}}e_{6}=e_{i}, \quad  1\leq i \leq 5
\end{equation}
and ${\nabla}_{e_{i}}e_{i}=0$, for $1\leq i\leq 5$, where
\begin{equation}\nonumber
e_{1}=\partial x_{1},\quad e_{2}=\partial x_{2},\quad e_{3}=\partial x_{3}, \quad e_{4}=\partial x_{4},\quad e_{5}=\partial x_{5}, \quad e_{6}=\partial x_{6}.
\end{equation}
Thus the statistical manifold $(\mathcal{M},\nabla,g)$ is of constant curvature. The scalar curvature is $-20$, hence it is an Einstein statistical manifold.
\end{example}

\section{Ricci-Bourguignon Solitons along Statistical Submersions}\label{sect_5}

This section deals with  Ricci-Bourguignon soliton of a statistical submersion $\psi:(\mathcal{M},\nabla, g)\longrightarrow \nolinebreak[3] (\mathcal{N},\hat{\nabla},\hat{g})$  between statistical manifolds and discuss the nature of the fibers of such submersion. \\
\indent
As a consequence of equations (\ref{e5}) to (\ref{e8}), for a statistical submersion, we obtain the following results:

\begin{theorem}\label{t1}
Let $\psi:(\mathcal{M},\nabla, g)\longrightarrow(\mathcal{N},\hat{\nabla},\hat{g})$  be a statistical submersion between statistical manifolds. Then  the vertical distribution $\mathcal{V}$ is parallel with respect to the connection $\nabla$ (resp. $\nabla^{\ast}$), if the horizontal parts $\mathcal{T}_{F}H$ (resp. $\mathcal{T}^{\ast}_{F}H$) and $\mathcal{A}_{X}F$
(resp. $\mathcal{A}^{\ast}_{X}F$) of (\ref{e5}) and (\ref{e7}) vanish identically for any $X,Y\in\Gamma \mathcal{H}(\mathcal{M})$ and $E,F\in\Gamma \mathcal{V}(\mathcal{M})$.

Similarly, the horizontal distribution $\mathcal{H}$ is parallel with respect to the connection $\nabla$ (resp. $\nabla^{\ast}$), if the vertical parts $\mathcal{T}_{F}X$ (resp. $\mathcal{T}^{\ast}_{F}X$) and $\mathcal{A}_{X}Y$ (resp. $\mathcal{A}^{\ast}_{X}Y$) of (\ref{e6}) and (\ref{e8}) vanish identically, for any $X,Y\in\Gamma \mathcal{H}(\mathcal{M})$ and $E,F\in\Gamma \mathcal{V}(\mathcal{M})$.
\end{theorem}

Considering $(M,g)$ as a Ricci-Bourguignon soliton, from (\ref{rb2}) we have
\begin{equation}\label{r1}
(\mathcal{L}_{V}g)(E,F)+2Ric(E,F)+(2\lambda-\rho R)g(E,F)=0,
\end{equation}
for any $E,F\in\Gamma \mathcal{V}(\mathcal{M})$. Using (\ref{c3}) and (\ref{1}), we have
\begin{equation}\label{r2}
\frac{1}{2}\left[g(\nabla_{E}V,F)+g(\nabla_{F}V,E)\right]+\bar{Ric}(E,F)-g(\mathcal{T}_{E}F,N^{\ast})
\end{equation}
\begin{equation}\nonumber
+\sum^{n}_{i=1}[g((\nabla_{X_{i}}\mathcal{T})({E},F),X_{i})+g(\mathcal{A}_{X_{i}}E,\mathcal{A}^{\ast}_{X_{i}}F)]-\sum^{m}_{j=1}g((\nabla^{\ast}_{E_{j}}\mathcal{A})({X_{i}},X),V) \end{equation}
\begin{equation}\nonumber
+\Big(\lambda-\rho \Big\{ \bar{R}- g(\mathcal{A}, \mathcal{A}^{\ast}) - g(N, N^{\ast})
\end{equation}
\begin{equation}\nonumber
- \hat{\delta}N-\hat{\delta}^{\ast}N^{\ast} - \bar{\delta}\sigma + \bar{\delta}^{\ast}\sigma+\left\|\sigma^{2}\right\|\Big\}\Big)g(E,F)=0,
\end{equation}
where $\left\{X_{i}\right\}_{1 \leq i \leq n }$ denotes an orthonormal basis of the horizontal distribution $\mathcal{H}$. Then, from Theorem \ref{t1} and equations (\ref{e2}), (\ref{e5}), we turn up the following expression:
\begin{equation}\label{r3}
\frac{1}{2}[\bar{g}(\bar{\nabla}_{E}V,F)+\bar{g}(\bar{\nabla}_{F}V,E)]+\bar{Ric}(E,F)+(\Lambda-\rho\bar {R})\bar{g}(E,F)=0,
\end{equation}
where $\Lambda= \lambda+\rho[g(\mathcal{A}, \mathcal{A}^{\ast}) + g(N, N^{\ast})+\hat{\delta}N+\hat{\delta}^{\ast}N^{\ast}]$.
%for any $E,F\in\Gamma \mathcal{V}(\mathcal{M})$.
Thus, we have the following:

\begin{theorem}\label{St2}
Let $(\mathcal{M},g,V,\lambda,\rho)$ be a Ricci-Bourguignon soliton with vertical potential vector field $V$ and let $\psi:(\mathcal{M},\nabla, g)\longrightarrow(\mathcal{N},\hat{\nabla},\hat{g})$ be a statistical submersion. If the vertical distribution $\mathcal{V}$ is parallel, then any fiber of the statistical submersion $\psi$ is a Ricci-Bourguignon soliton which satisfies (\ref{r3}).
\end{theorem}

Now, in a similar fashion, using (\ref{r3}), we state the following corollaries:

\begin{corollary}\label{Stc1}
Let $(\mathcal{M},g,V,\lambda,\rho=\frac{1}{2})$ be an Einstein soliton with vertical potential vector field $V$ and let $\psi:(\mathcal{M},\nabla, g)\longrightarrow(\mathcal{N},\hat{\nabla},\hat{g})$ be a statistical submersion. If the vertical distribution $\mathcal{V}$ is parallel, then any fiber of the statistical submersion $\psi$ is an Einstein soliton which satisfies
\begin{equation}\label{c1c}
\frac{1}{2}[\bar{g}(\bar{\nabla}_{E}V,F)+\bar{g}(\bar{\nabla}_{F}V,E)]+\bar{Ric}(E,F)+\left(\Lambda-\frac{\bar {R}}{2}\right)\bar{g}(E,F)=0.
\end{equation}
\end{corollary}

\begin{corollary}\label{Stc2}
Let $(\mathcal{M},g,V,\lambda,\rho=\frac{1}{2(n-1)})$ be a Schouten soliton  with vertical potential vector field $V$ and let $\psi:(\mathcal{M},\nabla, g)\longrightarrow(\mathcal{N},\hat{\nabla},\hat{g})$ be a statistical submersion. If the vertical distribution $\mathcal{V}$ is parallel, then any fiber of the statistical submersion $\psi$ is a Schouten soliton  which satisfies
\begin{equation}\label{c2c}
\frac{1}{2}[\bar{g}(\bar{\nabla}_{E}V,F)+\bar{g}(\bar{\nabla}_{F}V,E)]+\bar{Ric}(E,F)+\left(\Lambda-\frac{\bar {R}}{2(n-1)}\right)\bar{g}(E,F)=0.
\end{equation}
\end{corollary}

\begin{corollary}\label{Stc3}
Let $(\mathcal{M},g,V,\lambda,\rho=0)$ be a Ricci soliton with vertical potential vector field $V$ and let $\psi:(\mathcal{M},\nabla, g)\longrightarrow(\mathcal{N},\hat{\nabla},\hat{g})$ be a statistical submersion. If the vertical distribution $\mathcal{V}$ is parallel, then any fiber of the statistical submersion $\psi$ is a Ricci soliton which satisfies
\begin{equation}\label{c3c}
\frac{1}{2}[\bar{g}(\bar{\nabla}_{E}V,F)+\bar{g}(\bar{\nabla}_{F}V,E)]+\bar{Ric}(E,F)+\Lambda\bar{g}(E,F)=0.
\end{equation}
\end{corollary}

Also, for the dual case, we have the following theorem:

\begin{theorem}\label{St3}
Let $(\mathcal{M},g,V,\lambda,\rho)$ be a Ricci-Bourguignon soliton with vertical potential vector field $V$ and let $\psi:(\mathcal{M},\nabla^{\ast}, g)\longrightarrow(\mathcal{N},\hat{\nabla}^{\ast},\hat{g})$ be a statistical submersion. If the vertical distribution $\mathcal{V}$ is parallel, then any fiber of the statistical submersion $\psi$ is a Ricci-Bourguignon soliton which satisfies
\begin{equation}\label{Sr3}
\frac{1}{2}[\bar{g}(\bar{\nabla}^{\ast}_{E}V,F)+\bar{g}(\bar{\nabla}^{\ast}_{F}V,E)]+\bar{Ric}^{\ast}(E,F)+(\Lambda-\rho\bar {R}^{\ast})\bar{g}(E,F)=0.
\end{equation}
\end{theorem}

\begin{remark}\label{R1}
Similarly, we can also obtain the corresponding corollaries for the dual case.
\end{remark}

\begin{theorem}\label{TS2}
Let $(\mathcal{M},g,V,\lambda,\rho)$ be a Ricci-Bourguignon soliton with the potential vector field $V$ and let $\psi:(\mathcal{M},\nabla, g) \longrightarrow (\mathcal{N},\hat{\nabla},\hat{g})$ be a statistical submersion. If the horizontal distribution $\mathcal{H}$ is parallel, then the followings are fulfilled
\begin{enumerate}
\item{} if the potential vector field $V$ is vertical, then $(\mathcal{N},\hat{\nabla},\hat{g})$ is an Einstein manifold,
\item{} if the potential vector field $V$ is horizontal, then $(\mathcal{N},\hat{\nabla},\hat{g})$ is a Ricci-Bour\-gui\-gnon soliton with potential vector field $V^{'} = \psi_{\ast}V.$
\end{enumerate}
\end{theorem}

\begin{proof}
  Adopting (\ref{rb2}), (\ref{c4}) and (\ref{1}) we turn up
\begin{equation}\label{Sc4}
\frac{1}{2}\left[g(\nabla_{X}V,Y)+g(\nabla_{Y}V,X)\right]+\hat{Ric}(\hat{X},\hat{Y})+g(\nabla^{\ast}_{X}N^{\ast},Y) \end{equation}
\begin{equation}
\nonumber
-\sum^{m}_{j=1}g(\mathcal{T}_{E_{j}}{X},\mathcal{T}_{E_{j}}Y)+\sum^{n}_{i=1}g(\mathcal{A}_{X}{X_{i}},\mathcal{A}^{\ast}_{Y}{X_{i}})+\sum^{n}_{i=1}g((\nabla_{X_{i}}\mathcal{A})(X_{i}, X),Y)
\end{equation}
\begin{equation}
\nonumber
+\sum^{n}_{i=1}[g(\mathcal{A}_{X_i}X_{i}, A_{X}Y)-g(\mathcal{A}^{\ast}_{X}X_{i},\mathcal{A}^{\ast}_{X}X_{i})]
\end{equation}
\begin{equation}
\nonumber
+\Big(\lambda-\rho \Big\{ \hat{R}- 2\left\|\mathcal{A}\right\|^{2} - g(\mathcal{T}, \mathcal{T}^{\ast})
\end{equation}
\begin{equation}
\nonumber
-\hat{\delta}N -\hat{\delta}^{\ast}N^{\ast} - \bar{\delta}\sigma + \bar{\delta}^{\ast}\sigma -\left\|\sigma^{2}\right\|\Big\}\Big)g(X,Y) =0,
\end{equation}
where $\hat{X}$ and $\hat{Y}$ are $\psi$-related to $X$ and $Y$, respectively, for any $X,Y\in \Gamma\mathcal{H}(\mathcal{M})$.
Now, applying Theorem \ref{t1} to above equation (\ref{Sc4}), we turn up
\begin{equation}\label{Sc5}
\left[g(\nabla_{X}V,Y)+g(\nabla_{Y}V,X)\right]+2\hat{Ric}(\hat{X},\hat{Y})+(2\Lambda-\rho \hat{R})g(X,Y)=0,
\end{equation}
where $\Lambda=\lambda+\rho[2\left\|\mathcal{A}\right\|^{2} + g(\mathcal{T}, \mathcal{T}^{\ast})+ \hat{\delta}N +\hat{\delta}^{\ast}N^{\ast}]$.\\

\noindent
(1)\quad If the vector field $V$ is vertical, from (\ref{e7}) it follows that
\begin{equation}\label{Sc6}
\left[g(\mathcal{A}_{X}V,Y)+g(\mathcal{A}_{Y}V,X)\right]+2\hat{Ric}(\hat{X},\hat{Y})+(2\Lambda-\rho \hat{R})g(X,Y)=0.
\end{equation}
Since $\mathcal{H}$ is parallel, we find
\begin{equation}\label{Sc7}
\hat{Ric}(\hat{X},\hat{Y})+\left(\Lambda-\frac{\rho \hat{R}}{2}\right)g(X,Y)=0,
\end{equation}
which implies that $(\mathcal{N},\hat{\nabla},\hat{g})$ is Einstein.\\
\noindent
(2)\quad If the vector field $V$ is horizontal, from (\ref{Sc5}) we obtain
\begin{equation}\label{Sc8}
(\mathcal{L}_{V}g)(X,Y)+2\hat{Ric}(\hat{X},\hat{Y})+(2\Lambda-\rho \hat{R})g(X,Y)=0,
\end{equation}
which shows that the statistical manifold $(\mathcal{N},\hat{\nabla},\hat{g})$ is a Ricci-Bour\-guignon soliton with potential vector field $V^{'}=\psi_{\ast}V$.
\end{proof}

For the dual situation, we turn up the following theorem:

\begin{theorem}\label{TS3}
Let $(\mathcal{M},g,V,\lambda,\rho)$ be a Ricci-Bourguignon soliton with the potential vector field $V$ and let $\psi:(\mathcal{M},\nabla^{\ast}, g)\longrightarrow(\mathcal{N},\hat{\nabla}^{\ast},\hat{g})$ be a statistical submersion between statistical manifolds. If the horizontal distribution $\mathcal{H}$ is parallel, then  the followings are fulfilled:
\begin{enumerate}
\item{} if the potential vector field $V$ is vertical, then $(\mathcal{N},\hat{\nabla}^{\ast},\hat{g})$ is an Einstein manifold,
\item{} if the potential vector field $V$ is horizontal, then $(\mathcal{N},\hat{\nabla}^{\ast},\hat{g})$ is a Ricci-Bourguignon soliton with potential field $V^{'}=\psi_{\ast}V$.
\end{enumerate}
\end{theorem}

\section{Ricci-Bourguignon Solitons on Statistical Submersions with a Conformal Vector Field}\label{sect_6}

\begin{definition}\label{def1}
A smooth vector field $\zeta$ on a Riemannian manifold $(\mathcal{M},g)$ is said to a {\em conformal vector field} if there exists a smooth function $\varphi$ on $\mathcal{M}$ that satisfies \cite{conf}
\begin{equation}\label{D1}
\mathcal{L}_{\zeta}g=2\varphi g,
\end{equation}
where $\mathcal{L}_{\zeta}g$ is the Lie derivative of $\zeta$ with respect to $g$. If $\varphi=0$,  then $\zeta$ is called isometric (as well as  Killing vector field).
\end{definition}

\begin{theorem}\label{Dt1}
Let $(\mathcal{M},g,\zeta,\lambda,\rho)$ be a Ricci-Bourguignon soliton with conformal vector field $\zeta\in\nolinebreak[3] \Gamma T(\mathcal{M})$ and let $\psi:(\mathcal{M},\nabla, g)\longrightarrow(\mathcal{N},\hat{\nabla},\hat{g})$ be a statistical submersion. If the conformal vector field $\zeta$ is vertical and $(\varphi+\Lambda-\frac{\rho\bar{R}}{2})\neq 0$, then any fiber of the statistical submersion is Einstein with scalar curvature $\bar{R}=-\frac{m(\varphi+\Lambda)}{(1-\frac{m\rho}{2})}$.
\end{theorem}

\begin{proof}
Since $(\mathcal{M},g,\zeta,\lambda,\rho)$ is a Ricci-Bourguignon soliton, for any $E,F\in \Gamma T(\mathcal{M})$, adopting Theorem \ref{St2} and equation (\ref{r3}), as well, we have
\begin{equation}\label{D2}
[\bar{g}(\bar{\nabla}_{E}V,F)+\bar{g}(\bar{\nabla}_{F}V,E)]+2\bar{Ric}(E,F)+(2\lambda-\rho\bar {R})\bar{g}(E,F)=0.
\end{equation}
Now, involving (\ref{D2}), (\ref{D1}) and (\ref{1}) we have
\begin{equation}\label{D3}
\bar{Ric}(E,F)+\left(\varphi+\Lambda-\frac{\rho\bar {R}}{2}\right)\bar{g}(E,F)=0,
\end{equation}
where $\Lambda= \lambda+\rho[g(\mathcal{A}, \mathcal{A}^{\ast}) + g(N, N^{\ast})+\hat{\delta}N+\hat{\delta}^{\ast}N^{\ast}]$, which shows that the fiber of $\psi$ is Einstein.  Contracting (\ref{D3}), we get
\begin{equation}\label{D4}
\bar{R}=-\frac{m(\varphi+\Lambda)}{(1-\frac{m\rho}{2})}.
\end{equation}
\end{proof}

\begin{corollary}\label{DCt1}
Let $(\mathcal{M},g,\zeta,\lambda,\rho)$ be a Ricci-Bourguignon soliton with Killing vector field $\zeta\in\Gamma T(\mathcal{M})$ and let $\psi:(\mathcal{M},\nabla, g)\longrightarrow(\mathcal{N},\hat{\nabla},\hat{g})$ be a statistical submersion. If the Killing vector field $\zeta$ is vertical and $(\Lambda-\frac{\rho\bar{R}}{2})\neq 0$, then any fiber of the statistical submersion is also Einstein with scalar curvature $\bar{R}=-\frac{m\Lambda}{(1-\frac{m\rho}{2})}$.
\end{corollary}

Now, for dual case,  we also obtain the following:
\begin{theorem}\label{Dt2}
Let $(\mathcal{M},g,\zeta,\lambda,\rho)$ be a Ricci-Bourguignon soliton with conformal vector field $\zeta\in\nolinebreak[3]\Gamma T(\mathcal{M})$ and let $\psi:(\mathcal{M},\nabla^{\ast}, g)\longrightarrow(\mathcal{N},\hat{\nabla}^{\ast},\hat{g})$ be a statistical submersion. If the conformal vector field $\zeta$ is vertical and $(\varphi+\Lambda-\frac{\rho\bar{R}^{\ast}}{2})\neq 0$, then any fiber  of the statistical submersion is Einstein with scalar curvature $\bar{R}^{\ast}=-\frac{m(\varphi+\Lambda)}{(1-\frac{m\rho}{2})}$.
\end{theorem}

\begin{corollary}\label{DCt2}
Let $(\mathcal{M},g,\zeta,\lambda,\rho)$ be a Ricci-Bourguignon soliton with Killing vector field $\zeta\in\Gamma T(\mathcal{M})$ and let $\psi:(\mathcal{M},\nabla^{\ast}, g)\longrightarrow(\mathcal{N},\hat{\nabla}^{\ast},\hat{g})$ be a statistical submersion. If the  Killing vector field $\zeta$ is vertical and $(\varphi+\Lambda-\frac{\rho\bar{R}^{\ast}}{2})\neq 0$, then any fiber of the statistical submersion is also  Einstein with scalar curvature $\bar{R}^{\ast}=-\frac{m\Lambda}{(1-\frac{m\rho}{2})}$.
\end{corollary}

\begin{remark}\label{re2}
Now, we have important consequences of Theorem \ref{Dt1} and Theorem \ref{Dt2}. For particular values of $\rho$'s such as $\rho=\frac{1}{2}$, $\rho=\frac{1}{2(n-1)}$, $\rho=0$,  we can easily obtain the results for Einstein soliton, Schouten soliton and Ricci soliton, respectively.
\end{remark}

Using equation (\ref{D4}), we turn up the following:

\begin{theorem}\label{Dtt2}
If $\psi:(\mathcal{M},\nabla, g)\longrightarrow(\mathcal{N},\hat{\nabla},\hat{g})$  is a statistical submersion with vertical conformal vector field $\zeta\in\Gamma T(\mathcal{M})$ and any fiber admits a Ricci-Bourguignon soliton $(\mathcal{M},g,\zeta,\lambda,\rho)$, then the Ricci-Bour\-guignon soliton of any fiber with scalar curvature $\bar{R}$  is respectively expanding, steady and shrinking according as
\begin{enumerate}
\item[(i)] $\bar{R}\left(\frac{\rho}{2}-\frac{1}{m}\right)> \varphi$,
\item[(ii)] $\bar{R}\left(\frac{\rho}{2}-\frac{1}{m}\right)= \varphi$,
\item[(iii)] $\bar{R}\left(\frac{\rho}{2}-\frac{1}{m}\right)< \varphi$.
\end{enumerate}
\end{theorem}

\begin{theorem}\label{Dtt1}
If $\psi:(\mathcal{M},\nabla^{\ast}, g)\longrightarrow(\mathcal{N},\hat{\nabla}^{\ast},\hat{g})$  is a statistical submersion with vertical conformal vector field $\zeta\in\Gamma T(\mathcal{M})$ and any fiber admits a Ricci-Bourguignon soliton $(\mathcal{M},g,\zeta,\lambda,\rho)$, then the Ricci-Bour\-guignon soliton of any fiber  with scalar curvature $\bar{R}^{\ast}$ is respectively expanding, steady and shrinking according as
\begin{enumerate}
\item[(i)] $\bar{R}^{\ast}\left(\frac{\rho}{2}-\frac{1}{m}\right)> \varphi$,
\item[(ii)] $\bar{R}^{\ast}\left(\frac{\rho}{2}-\frac{1}{m}\right)= \varphi$,
\item[(iii)] $\bar{R}^{\ast}\left(\frac{\rho}{2}-\frac{1}{m}\right)< \varphi$.
\end{enumerate}
\end{theorem}

\section{Gradient Type Ricci-Bourguignon Solitons on Statistical Submersions}\label{sect_7}

This section deals with the gradient Ricci-Bourguignon soliton of statistical submersions with the potential vector field $V=grad (\Psi)$, where $\Psi$ is a smooth function on $\mathcal{M}$.\\
\indent
 Consider the equation (\ref{r3}). Under condition of Theorem \ref{St2}, we get
\begin{equation}\label{Gr1}
\bar{Ric}(E,F)=-\left(\Lambda-\frac{\rho\bar {R}}{2}\right)\bar{g}(E,F)-\frac{1}{2}[\bar{g}(\bar{\nabla}_{E}V,F)+\bar{g}(\bar{\nabla}_{F}V,E)].
\end{equation}
\indent
Contracting (\ref{Gr1}), we get
\begin{equation}\label{Gr2}
\bar{R}=-m\left(\Lambda-\frac{\rho\bar {R}}{2}\right)-div(V).
\end{equation}

Now, we obtain the following results:

\begin{theorem}\label{GTr1}
 Let $\psi:(\mathcal{M},\nabla, g)\longrightarrow(\mathcal{N},\hat{\nabla},\hat{g})$  be a statistical submersion  with vertical potential field $V=grad(\Psi)$. If the vertical distribution $\mathcal{V}$ is parallel, then any fiber of the statistical submersion $\psi$ admits gradient Ricci-Bourguignon soliton and the Poisson equation satisfied by $\Psi$ becomes
\begin{equation}\label{Gr4}
\Delta(\Psi)=\left[m\Lambda-\bar{R}\left(1-\frac{m\rho}{2}\right)\right].
\end{equation}
\end{theorem}

For the dual case, we turn up

\begin{theorem}\label{GTr2}
Let $\psi:(\mathcal{M},\nabla^{\ast}, g)\longrightarrow(\mathcal{N}^{\ast},\hat{\nabla},\hat{g})$  be a statistical submersion  with vertical potential field $V=grad(\Psi)$. If the vertical distribution $\mathcal{V}$ is parallel, then any fiber of the statistical submersion $\psi$ admits gradient Ricci-Bourguignon soliton and the Poisson equation satisfied by $\Psi$ becomes
\begin{equation}\label{Gr5}
\Delta^{\ast}(\Psi)=\left[m\Lambda-\bar{R}^{\ast}\left(1-\frac{m\rho}{2}\right)\right].
\end{equation}
\end{theorem}

Now, we underline some useful applications of Theorem \ref{GTr1}.\\

\noindent
{\bf Significance of Poisson equation in Physics}. The general theory of solution of Poisson equation is known as {\textit{potential theory}} and the solution of Poisson equations are harmonic functions, which are important in different branches of physics such as electrostatics, gravitation and fluid dynamics. In modern physics, there are two fundamental forces of the nature known at the time, namely, gravity and  electrostatics forces, could be modeled using functions called the gravitational potential and electrostatics potential both of which satisfy Laplace equations. For example, consider the phenomena: if $\Psi$ is the gravitational filed, $\rho$ is the mass density and $G$ is the gravitational constants, the Gauss's law of gravitation   is
\begin{equation}\label{l1}
\nabla\Psi=-4\pi G\rho.
\end{equation}
In case of  gravitational field, $\Psi$ is conservative and can be expressed as the negative gradient of gravitational potential, that is, $\Psi=-grad (f)$. Then by the Gauss's law of gravitation, we have
\begin{equation}\label{l2}
\nabla^{2}f=4\pi G\rho.
\end{equation}
 This physical phenomena is directly identical to the Theorem \ref{GTr1}, equation (\ref{Gr4}) and its respective dual, which is a Poisson equation with potential vector field of gradient type, that is, $V=grad(\Psi)$.

 A function $\Psi:\mathcal{M}\longrightarrow \mathbb{R}$ is said to be harmonic if $\Delta\Psi=0$, where $\Delta$ is the Laplacian operator in $\mathcal{M}$ \cite{Yau} and  we turn up the following results:

\begin{theorem}\label{GTr3}
Let $\psi:(\mathcal{M},\nabla, g)\longrightarrow(\mathcal{N},\hat{\nabla},\hat{g})$ be a statistical submersion with vertical potential vector field $V=grad(\Psi)$ and $\Psi$ a harmonic function. If the vertical distribution $\mathcal{V}$ is parallel, then any fiber of the statistical submersion $\psi$  admits expanding, steady or shrinking gradient Ricci-Bourguignon soliton according as $ \frac{\rho}{2}>\frac{1}{m}$, $ \frac{\rho}{2}=\frac{1}{m}$ and $ \frac{\rho}{2}<\frac{1}{m}$, respectively.
\end{theorem}

Also, for its dual.

\begin{theorem}\label{GTr4}
Let $\psi:(\mathcal{M},\nabla^{\ast}, g)\longrightarrow(\mathcal{N},\hat{\nabla}^{\ast},\hat{g})$  be a statistical submersion with vertical potential field $V=grad(\Psi)$ and $\Psi$ a harmonic function. If the vertical distribution $\mathcal{V}$ is parallel, then any fiber of the statistical submersion $\psi$  admits expanding, steady or shrinking gradient Ricci-Bourguignon soliton according as $ \frac{\rho}{2}>\frac{1}{m}$, $ \frac{\rho}{2}=\frac{1}{m}$ and $ \frac{\rho}{2}<\frac{1}{m}$, respectively.
\end{theorem}

\begin{example}
	Let $\psi : (\mathbb{R}^{6},\nabla,g) \rightarrow (\mathbb{R}^{3},\hat{\nabla},\hat{g})$  be a statistical submersion between statistical manifolds mentioned in Example (\ref{ex1}) defined by $$\psi(x_{1},...,x_{6})=(y_1,y_2,y_3),$$ where
	$$y_1=\frac{x_1+x_2}{\sqrt{2}}, \ y_2=\frac{x_3+x_4}{\sqrt{2}} \ \textrm{ and } \ y_3=\frac{x_5+x_6}{\sqrt{2}}.$$
	
	Then, the Jacobian matrix of $\psi$ is as follows
		\[
\begin{bmatrix}
	\frac{1}{\sqrt{2}}&\frac{1}{\sqrt{2}}&0&0&0&0\\
0&0&\frac{1}{\sqrt{2}}&\frac{1}{\sqrt{2}}&0&0\\
0&0&0&0&\frac{1}{\sqrt{2}}&\frac{1}{\sqrt{2}}\\
\end{bmatrix}
.
\]
	
Since the rank of matrix is equal to 3, which is equal to the dimension of target space $(\mathbb{R}^{3},\hat{\nabla},\hat{g})$. On the other hand, we easily see that $\psi$ satisfies the conditions of statistical submersion. That means $\psi$ is a statistical submersion. A straight forward computation yields
	\begin{align*}
	\mathcal{V}(\mathcal{M})&=span \Big\{V_{1}=\frac{1}{\sqrt{2}}(-\partial x_{1}+\partial x_{2}), V_{2}=\frac{1}{\sqrt{2}}(-\partial x_{3}+\partial x_{4}),
\\&V_3=\frac{1}{\sqrt{2}}(-\partial x_{5}+\partial x_{6})\Big\}, \\
	\mathcal{H}(\mathcal{M})&=span \Big\{H_{1}=\frac{1}{\sqrt{2}}(\partial x_{1}+\partial x_{2}), H_{2}=\frac{1}{\sqrt{2}}(\partial x_{3}+\partial x_{4}), \\&H_3=\frac{1}{\sqrt{2}}(\partial x_{5}+\partial x_{6})\Big\}.
	\end{align*}
	Also by direct computations yields
	$$\psi_{*}(H_1)=\partial y_1, \psi_{*}(H_2)=\partial y_2 \ \textrm{and}\ \psi_{*}(H_3)=\partial y_3.$$
	Hence, it is easy to see that
	$$g_{\mathbb{R}^{6}}(H_i,H_i)=\hat{g}_{\mathbb{R}^{3}}(\psi_{*}(H_i),\psi_{*}(H_i)), \ i=1,2,3.$$
	Thus $\psi$ is a statistical submersion.\\
	\indent

	Now, we show that the fiber $\psi_{\ast}$, i.e, $\mathcal{V}(\mathcal{M})$ (vertical space) admits Ricci-Bour\-gui\-gnon soliton. Therefore, for  the vertical space, we have
	\begin{equation}\nonumber
	\bar{Rie}(V_{1},V_{2})V_{1}=-2V_{2}, \quad \bar{Rie}(V_{1},V_{2})V_{2}=2V_{1},\quad \bar{Rie}(V_{1},V_{3})V_{1}=-2V_{3}
	\end{equation}
	\begin{equation}\nonumber
	\bar{Rie}(V_{1},V_{2})V_{3}=V_{1}, \quad \bar{Rie}(V_{2},V_{3})V_{3}=V_{2},\quad \bar{Rie}(V_{2},V_{3})V_{2}=V_{2}.
	\end{equation}
	%{S}(V_{i},V_{j})= \left(\begin{array}{ccc}\label{x03}
	%2&0&0\\\nonumber
%0&2&0&\\\nonumber
%0&0&1&\\   \nonumber
%\end{array}\right).
%\end{equation}

\[
\bar{Ric}(V_{i},V_{j})=
\begin{bmatrix}
	2&0&0\\
0&2&0\\
0&0&1\\
\end{bmatrix}
\quad \quad i,j=1,2,3.
\]

\begin{equation}\nonumber
\bar{R}=tr(\hat{Ric})=5.
\end{equation}
Using (\ref{r3}), we obtain  $\lambda=5\rho-2$. Therefore ($\mathcal{V}(\mathcal{M}),\bar{\nabla},\hat{g})$ is admitting an expanding, shrinking or steady {\em Ricci-Bourguignon soliton} according $5\rho>2$,  $5\rho<2$ or $5\rho=2$, respectively. Moreover, for particular values of $\rho$ we turn up the following cases:\\
{\bf Case 1.}\quad For $\rho=\frac{1}{2}$ we obtain $\lambda=\frac{1}{2}$. Therefore $(\mathcal{V}(\mathcal{M}),\bar{\nabla}, \hat{g})$ is admitting the expanding {\em Einstein soliton}.\\
{\bf Case 2.}\quad For $\rho=\frac{1}{(2n-1)}$ we obtain $\lambda=-1$. Therefore ($\mathcal{V}(\mathcal{M}), \bar{\nabla}, \hat{g})$ is admitting the shrinking {\em Schouten soliton}.\\
{\bf Case 3.}\quad For $\rho=0$ we obtain $\lambda=-2$. Therefore ($\mathcal{V}(\mathcal{M}), \bar{\nabla}, \hat{g})$ is admitting the shrinking {\em Ricci soliton}.

\end{example}

%
%
%\section*{Acknowledgement} % A simple way to write an acknowledgement
%The authors express their gratitude to the XXX Institute at YYY for
%their hospitality, etc
%
%
%\bibliographystyle{\mmnbibstyle}
%\bibliography{\jobname}

\end{document}